\documentclass[letterpaper,11pt]{article}
\usepackage{color}
\usepackage[usenames,dvipsnames]{xcolor}
\usepackage[small]{titlesec}
\usepackage{theorem,amsmath,amssymb,amscd, verbatim}
\usepackage[all]{xy}
\usepackage{booktabs}
\usepackage[hyperfootnotes=false, colorlinks, linkcolor={blue}, citecolor={magenta}, filecolor={blue}, urlcolor={blue}]{hyperref}
\usepackage[overload]{textcase} 
\theoremstyle{change}
\allowdisplaybreaks
\newcommand{\leftexp}[2]{{\vphantom{#2}}^{#1}%
      \kern-\scriptspace%
      {#2}}

\renewcommand{\Im}{\mathrm{Im}}
\renewcommand{\Re}{\mathrm{Re}}
\renewcommand{\H}{\mathbb H}

\newcommand{\A}{{\mathbb A}}

\newcommand{\Q}{{\mathbb Q}}
\newcommand{\Z}{{\mathbb Z}}

\newcommand{\R}{{\mathbb R}}
\newcommand{\C}{{\mathbb C}}
\newcommand{\bs}{\backslash}

\renewcommand{\O}{{\mathcal O}}
\newcommand{\GL}{{\rm GL}}

\newcommand{\GSp}{{\rm GSp}}
\newcommand{\Sp}{{\rm Sp}}

\newcommand{\Aut}{{\rm Aut}}

\newcommand{\sgn}{{\rm sgn}}

\newcommand{\sym}{{\rm sym}}

\newcommand{\mat}[4]{{\setlength{\arraycolsep}{0.5mm}\left[
\begin{smallmatrix}#1&#2\\#3&#4\end{smallmatrix}\right]}}
\newcommand{\qed}{\hspace*{\fill}\rule{1ex}{1ex}}
\newcommand{\forget}[1]{}

\def\qdots{\mathinner{\mkern1mu\raise0pt\vbox{\kern7pt\hbox{.}}\mkern2mu
\raise3.4pt\hbox{.}\mkern2mu\raise7pt\hbox{.}\mkern1mu}}

\newenvironment{proof}{\vspace{1ex}\noindent{\it Proof.}\hspace{0.1em}}
	{\hfill\qed\vspace{2ex}}
\newenvironment{bsmallmatrix}{\left[\begin{smallmatrix}}{\end{smallmatrix}\right]}

\newtheorem{lemma}{Lemma.}[section]
\newtheorem{theorem}[lemma]{Theorem.}

\newtheorem{proposition}[lemma]{Proposition.}

\newtheorem{remark}[lemma]{Remark.}

\begin{document}

\bibliographystyle{plain}
\thispagestyle{empty}
\begin{center}
     {\bf\Large Integrality and cuspidality of pullbacks of nearly holomorphic Siegel Eisenstein series}

 \vspace{3ex}
 Ameya Pitale, Abhishek Saha, Ralf Schmidt
\end{center}

\begin{abstract}We study nearly holomorphic Siegel Eisenstein series of general levels and characters on $\H_{2n}$, the Siegel upper half space of degree $2n$. We prove that   the Fourier coefficients of these Eisenstein series (once suitably normalized) lie in the ring of integers of $\Q_p$ for all sufficiently large primes $p$. We also prove that the pullbacks of these Eisenstein series to $\H_n \times \H_n$ are cuspidal under certain assumptions.
\end{abstract}

\tableofcontents

\section{Introduction}Let $\H_n$ denote the Siegel upper half space of degree $n$ on which the group $\Sp_{2n}(\R)$ acts in the usual way. In the arithmetic and analytic theory of $L$-functions attached to automorphic forms on $\GSp_{2n}$, a major role is played by pullbacks (i.e., restrictions) of Siegel Eisenstein series from $\H_{2n}$ to $\H_n \times \H_n$. The importance of such pullbacks has been realized at least since the early 1980s when a remarkable integral representation (the pullback formula) for the standard $L$-functions attached to Siegel cusp forms  was discovered  by Garrett \cite{gar83}. Subsequently, Shimura obtained a similar formula  in a wide variety of contexts (including other groups). We refer to Shimura's books \cite{shibook1, shibook2} for further details.
We note also that this method of obtaining integral representations for $L$-functions using pullbacks of Eisenstein series  may be viewed as a special case of the ``doubling method" of Piatetski-Shapiro and Rallis \cite{psrallis83, psrallis87}.

In the last three decades, the pullback formula has been used to prove a host of results related  to Siegel cusp forms and their $L$-functions. We refer the reader to the introduction of our recent paper \cite{PSS17} for a more detailed history. As explained in \cite{PSS17}, most of the previous results involved significant restrictions on the levels or archimedean parameters of the Siegel cusp forms involved. These were removed in \cite{PSS17} to a large extent, where we obtained an explicit integral representation for the twisted standard $L$-function attached to a vector-valued Siegel cusp form of arbitrary level and archimedean parameter.  More precisely,  for positive integers $N$, $k$, and a Dirichlet character $\chi$ mod $N$,  define the Siegel Eisenstein series
\begin{align}\label{eisensteinserieseq}
E_{k,N}^{\chi}(Z,s; 1)=\sum_{\gamma = \mat{A}{B}{C}{D} \in (P_{4n}(\Q)\cap \Sp_{4n}(\Z))\bs \Gamma_{0,4n}(N)}\chi^{-1}(\det(A))\det(\Im(\gamma Z))^{s}\,J(\gamma,Z)^{-k}.
\end{align}
Above, we have $s \in \C$ and $Z \in \H_{2n}$, $P_{4n}$ denotes the Siegel parabolic subgroup of $\Sp_{4n}$, and $\Gamma_{0,4n}(N)$ denotes the Siegel congruence subgroup of $\Sp_{4n}(\Z)$ of level $N$, consisting of matrices whose lower left block is congruent to zero modulo $N$. The series \eqref{eisensteinserieseq} converges absolutely and uniformly for $\Re(s)$ sufficiently large, and the function $E_{k,N}^{\chi}(Z,s; 1)$ is defined by analytic continuation outside this region.

More generally, for each element $h \in \Sp_{4n}(\hat{\Z})$, define $E_{k,N}^{\chi}( Z, s; h):=E_{k,N}^{\chi}(Z,s;1)\big|_{k}h.$ Then, for a smooth modular form\footnote{In general, the (scalar-valued) function $F$ is not holomorphic, but it can be obtained by applying suitable differential operators on a  holomorphic (vector-valued) Siegel cusp form of degree $n$.} $F$ of level $N$ and weight $k$ on $\H_n$ associated to a particular choice of archimedean vector inside an automorphic representation $\pi$ of $\GSp_{2n}(\A)$, and for a certain $Q_\tau \in \Sp_{4n}(\hat{\Z})$, the main result of \cite{PSS17} was an identity of the form
\begin{equation}\label{e:formulaold}
 \int\limits_{ \Gamma_{2n}(N) \bs \H_n}E_{k,N}^{\chi}\!\left(\mat{Z_1}{}{}{Z_2}, \frac{s+n-k}2; Q_\tau \right) \bar{F}(Z_1)\,dZ_1 \approx \frac{L(s,\pi \boxtimes \chi)}{L(s+n, \chi)\prod_{j=0}^{n-1}L(2s+2r, \chi^2)} F(Z_2)
\end{equation} where $L(s,\pi \boxtimes \chi)$ is the degree $2n+1$ $L$-function on $\GSp_{2n} \times \GL_1$, the element $\mat{Z_1}{}{}{Z_2}$ of $\H_{2n}$ is obtained from the diagonal embedding of $\H_n \times \H_n$,  and the symbol $\approx$ indicates that the two sides are equal up to some (well-understood) explicit factors.

The Siegel Eisenstein series $E_{k,N}^{\chi}(Z,s; h)$ on $\H_{2n}$ defined above corresponds to a special case of the degenerate Eisenstein series on $\GSp_{4n}(\A)$ in the sense of Langlands and has been studied previously by many authors, including Shimura \cite{shi83, shibook1}, Feit \cite{feit86}, Garrett \cite{gar83, gar2}, B\"ocherer \cite{boch84, boch1985}, and Brown \cite{Brown2007, brown11}. It is a well-known fact, going back to Shimura \cite{shi83}, that if either $k > n+1$, or $k=n+1$ and $\chi^2 \neq 1$, then the Eisenstein series  $E_{k,N}^{\chi}(Z,0; h)$ (i.e., at the special point $s=0$) represents a \emph{holomorphic} Siegel modular form of weight $k$ and degree $2n$. (We note here that this fact follows immediately in the range $k\ge 2n+1$ where the series converges absolutely, but is more delicate for  $n+1 \le k \le 2n$).

The goal of this paper is to better understand some key aspects of this Eisenstein series at certain negative integer values  $s=-m_0$ for which the Eisenstein series becomes a \emph{nearly holomorphic} Siegel modular form. We recall that  a nearly holomorphic Siegel modular form of degree $n$ is a function on $\H_n$ that has the same definition as a holomorphic Siegel modular form of degree $n$, except that we weaken the holomorphy condition to the condition of being a polynomial in the entries of $\Im(Z)^{-1}$ over the ring of holomorphic functions.
If $0 \le m_0 \le \frac{k-n-1}2$ is an integer, then  $E_{k,N}^{\chi}(Z,-m_0; h)$ represents a nearly holomorphic Siegel modular form of weight $k$ and degree $2n$ (except in the special case $\chi^2 = 1$, $m_0 =\frac{k-n-1}2$).

From classical results of Feit and Shimura it follows (see \cite[Prop. 6.8]{PSS17}) that for each $m_0$ as above, and each $h \in \Sp_{4n}(\hat{\Z})$, the Fourier coefficients of $E_{k,N}^{\chi}(Z,-m_0; h)$ are \emph{algebraic numbers} which behave nicely under the action of $\Aut(\C)$. This was an important ingredient in our paper \cite{PSS17} for proving the expected algebraicity result for the special values
of $L(s,\pi \boxtimes \chi)$ in the case of $n=2$.
Motivated by the Bloch-Kato conjectures, we would like to prove more refined arithmetic results on these special $L$-values. In order to achieve such results, the following questions become crucial to answer:

\begin{enumerate}

\item \label{q1} Is the nearly holomorphic modular form $E_{k,N}^{\chi}\left(\mat{Z_1}{}{}{Z_2}, -m_0; Q_\tau\right)$  \emph{cuspidal} in both variables $Z_1, Z_2$ when $N>1$?

\item \label{q2} Can one say anything about the primes dividing the denominators of the (algebraic) Fourier coefficients  of  $E_{k,N}^{\chi}( Z, -m_0; h)$; in particular, can we prove that these coefficients are \emph{$p$-integral} outside a finite specified set of primes $p$?
\end{enumerate}

A positive answer to the former question will allow us to write $E_{k,N}^{\chi}\left(\mat{Z_1}{}{}{Z_2}, -m_0; Q_\tau\right)$ as an explicit bilinear sum over nearly holomorphic Siegel \emph{cusp} forms that are Hecke eigenforms. Combining this with a positive answer to the latter question will lead to new results on congruences between Hecke eigenvalues of Siegel cusp forms, which will open up many avenues for future exploration. We will discuss some of these potential applications at the end of this introduction.

 Previous work on the two questions posed above has been largely restricted to the holomorphic case  $m_0=0$. For $m_0=0$, Question \ref{q1} was answered affirmatively by Garrett \cite{gar2}, which allowed him to deduce important results on the arithmeticity of ratios of Petersson inner products associated to Siegel cusp forms.  In the direction of Question \ref{q2}, a product formula for the Fourier coefficients of the Eisenstein series was written down by Shimura \cite{shibook1}, which was subsequently used by Brown \cite[Theorem 4.4]{Brown2007} for the value $m_0 = \frac{2n+1}{2}-k$ (which is essentially equivalent to $m_0=0$ via the functional equation of the Eisenstein series)  to give a set of primes for which Question \ref{q2} has an affirmative answer (however, see Remark \ref{rem:brown}).
 Other results in the direction of Question \ref{q2} for $m_0=0$,  i.e., giving $p$-integrality results for holomorphic Siegel Eisenstein series, include the works of B\"ocherer \cite{boch84}, B\"ocherer--Schmidt \cite{bocherer-schmidt}, and  B\"ocherer--Dummigan--Schulze-Pillot \cite{BDSP}.

So far, there appears to have been relatively little work on the above two questions in the rest of the \emph{nearly holomorphic} range, i.e., when $0 < m_0 \le \frac{k-n-1}2 $, except in certain full level cases. To put the importance of the nearly holomorphic range in context, we remark that the holomorphic case $m_0=0$ corresponds (via the pullback formula) to at most one special $L$-value for the standard $L$-function. In contrast, the entire nearly holomorphic range for $m_0$ allows us to access \emph{all} the special $L$-values using \eqref{e:formulaold}.

In this paper, we prove the following theorem, which gives an affirmative answer to both questions posed above in a general setup.

\begin{theorem}\label{main-Eis-arith-thm}
 Let $N$ be a positive integer. Let $\chi=\otimes\chi_p$ be a finite order character of $\Q^\times\backslash\A^\times$ whose conductor divides $N$. Let $k$ be a positive integer such that $\chi_\infty = {\rm sgn}^k$. For $h \in \Sp_{4n}(\hat{\Z})$ and $m_0 \in \Z$, let
 $$
  E(Z;h) := \pi^{n+n^2-(2n+1)k+(2n+2)m_0}\Lambda^N\left(\frac{k-2m_0}2\right) E_{k,N}^{\chi}(Z,-m_0; h),
 $$
 where
 \begin{equation}
  \Lambda^N(s) := L^N(2s, \chi) \prod\limits_{i=1}^{n}L^N(4s-2i, \chi^2)
 \end{equation}
 and the notation $L^N$ in the $L$-functions for the Dirichlet characters means that we exclude the local factors at primes dividing $N$ or $\infty$. Then the following hold.
\begin{enumerate}
\item Suppose that $0 \le m_0 \le \frac{k-2n-2}{2}$ is an integer and that $N>1$. Then the function
 $
 E\left(\mat{Z_1}{}{}{Z_2};Q_\tau\right)
 $
 is cuspidal in each variable $Z_1$, $Z_2$. Here, $\tau\in\hat\Z^\times$, and $Q_\tau$ is defined in \eqref{Qtaueq}.

\item Suppose that $0 \le m_0 \le \frac{k-n-1}{2}$ is an integer.
  If $m_0 = \frac{k-n-1}{2}$, assume further that $\chi^2 \neq 1$. Let $p \nmid 2N$ be a prime such that $p \ge 2k$. Write out the usual Fourier expansion
$$E(Z; h) =  \sum_{\substack{S \in \frac1N M_{2n}^{\sym}(\Z) \\ S\ge 0}}a_E\left((2 \pi Y)^{-1} ; S;h \right)e^{2\pi i\,{\rm tr}(SZ)},$$  where $a_E(X; S;h)$ is a polynomial in the entries of the matrix $X$. Then, for any choice of isomorphism $\overline{\Q}_p \simeq \C$, the coefficients of the polynomial $a_E( - ; S;h)$
 lie in  the ring of integers of $\overline{\Q}_p$.
 \end{enumerate}
 \end{theorem}
\begin{remark}In the case $m_0 = \frac{k-n-1}{2}$ (which only occurs when $k \equiv n+1 \pmod{2}$), Theorem \ref{main-Eis-arith-thm} cannot handle the case $\chi^2 = 1$. This is due to the fact that the normalization of the Eisenstein series corresponding to this point involves the factor $L(1, \chi^2)$ which has a pole when $\chi^2=1$. (This can be seen directly from \eqref{b(h)-formula} in the case $m_0=0, k=n+1$  by looking at the Fourier coefficient $b(h)$ when  $\rho_h=\chi$). Consequently the required arithmetic results for the Eisenstein series are unavailable in this case.
The restriction $p \ge 2k$ in Theorem \ref{main-Eis-arith-thm} comes from the denominator present for the action of a certain differential operator; see \eqref{step4eq2}.

\end{remark}
We now discuss further the significance of this result. Similar results in the case $m_0=0$ have been previously used by various authors to show that primes dividing the denominators or numerators of certain $L$-values are in fact \emph{congruence primes} for certain lifted Siegel modular forms. In this context, we note the papers of Brown \cite{Brown2007, brown11}, Agarwal--Brown  \cite{brown-agarwal},  Katsurada \cite{katsu08, katsu15}, B\"ocherer--Dummigan--Schulze-Pillot \cite{BDSP} and Agarwal--Klosin \cite{AK}, all of which obtain congruence primes for lifted Siegel cusp forms to give evidence towards the \emph{Bloch-Kato conjecture} for appropriate $L$-functions. More generally, the method of using congruence primes for lifts was inaugurated by Ribet's famous proof of the converse of Herbrand's theorem. Following Wiles, Skinner-Urban and others, this idea has proved tremendously influential (e.g.,  to the main conjectures of Iwasawa theory). In a slightly different direction, there is recent work of  Tobias Berger and Krzysztof Klosin \cite{berger-klosin}  who have developed a new approach to proving non-trivial cases of the \emph{paramodular conjecture} using such congruences for paramodular Saito-Kurokawa lifts. We are  currently working on a  project to obtain a general result for congruence primes for $\GSp_4$, extending the works mentioned above. Theorem \ref{main-Eis-arith-thm} will be a key piece of this future generalization.

We now explain briefly the steps in our proof of Theorem \ref{main-Eis-arith-thm}. The basic strategy for proving part 1 of the theorem, i.e., the cuspidality of the pullback, goes back to Garrett \cite{gar2}. The key point here is that we assume $N>1$, and translate the Eisenstein series by $Q_\tau$, which  ensures that the pullback of the Eisenstein series unwinds to a sum over exactly \emph{one} double-coset representative. (This step also requires us to assume that the series defining $E_{k,N}^{\chi}(Z,-m_0; Q_\tau)$ is absolutely convergent, which is why we need to assume $m_0 \le \frac{k-2n-2}{2}.$) To deduce cuspidality from here, we show that the $S$'th Fourier coefficient of each term in the resulting sum vanishes unless $S$ is positive definite. For this, we reduce the problem to the vanishing of a real integral, which we prove by a direct computation; see Section \ref{s:cuspidal} for details. This last step of ours is quite different from Garrett's proof, who was able to directly use a computation of Siegel to show the vanishing of the integral; this is no longer available in our setup.

Our proof of part 2 of Theorem \ref{main-Eis-arith-thm}, i.e., the $p$-integrality of the Fourier coefficients, proceeds via several steps.
The first step, which is the most involved part of the argument, is done in Section \ref{s:initial} and treats the case $m_0=0$, $h = \mat{0}{I_{2n}}{-I_{2n}}{0}$. For this, we combine Shimura's work which expresses a general Fourier coefficient as an explicit Euler product of local Siegel series, explicit formulas \cite{boch84, kitaoka84} for these local Siegel series at bad primes, and classical facts on the $p$-integrality of Bernoulli numbers. The second step, carried out in Section \ref{s:qexp}, extends this to general $h$ via the powerful $q$-expansion principle (due to Chai-Faltings \cite{ChaiFaltings1990} in this setting). As a final step (carried out in Section \ref{s:final}), we treat the case of general $m_0$, using the Maass weight raising differential operator, and results due to Panchishkin \cite{Pan2005} and Shimura \cite{shibook2}. We remark that these steps do not require us to assume that the series defining $E_{k,N}^{\chi}(Z,-m_0; h)$ is absolutely convergent, and therefore, for part~2 of Theorem \ref{main-Eis-arith-thm}, we can let $m_0$ vary over the full range of integers that produces a nearly holomorphic modular form.
\subsection*{Acknowledgements}A.S. acknowledges the support of the Leverhulme Trust Research Project Grant RPG-2018-401. We also thank the London Mathematical Society for a Scheme 4 `Research in Pairs' grant (Ref: 41850) which supported a visit of A.P. to A.S. during which part of this paper was written. We would like to thank the referees for their comments and suggestions.

\subsection*{Notations}
For a positive integer $n$, let $G_{2n}$ be the algebraic group $\GSp_{2n}$, defined for any commutative ring $R$ by
$$
 G_{2n}(R)=\{g\in\GL_{2n}(R)\,:\,^tgJ_ng=\mu_n(g)J_n,\:\mu_n(g)\in R^\times\},\qquad J_n=\mat{}{I_n}{-I_n}{},
$$
where $I_n$ is the $n\times n$ identity matrix. The symplectic group $\Sp_{2n}$ consists of those elements $g\in G_{2n}$ for which the multiplier $\mu_n(g)$ is $1$. Let $\H_n$ be the Siegel upper half space of degree $n$, consisting of all complex, symmetric $n\times n$ matrices $X+iY$ with $X,Y$ real and $Y$ positive definite. The group $\Sp_{2n}(\R)$ acts transitively on $\H_n$ in the usual way (see Chapter I of \cite{klingen}). For each $n$, we let $I$ denote the element $iI_n \in \H_n$, and for $g=\mat{A}{B}{C}{D}\in\Sp_{2n}(\R)$, $Z \in \H_n$, we let $J(g, Z) =  \det(CZ+D)$.

For each discrete subgroup $\Gamma$ of $\Sp_{2r}(\Q)$, let $N_k(\Gamma)$ (resp.\ $N_k(\Gamma)^\circ$)  be the space of nearly holomorphic Siegel modular forms (resp.\ cusp forms) of weight $k$ for $\Gamma$. Similarly, we let $M_k(\Gamma)$ (resp. $S_k(\Gamma)$)  be the space of holomorphic Siegel modular forms (resp.\ cusp forms) of weight $k$ for $\Gamma$. We let $\Gamma_{2r}(N)$ denote the principal congruence subgroup of level $N$.

Given a subring $R$ of $\C$, we define the $R$-module $N_k(\Gamma, R)$ to consist of all the elements in $N_k(\Gamma)$ whose Fourier coefficients lie in $R$. Concretely, an element $F \in  N_k(\Gamma)$ has a Fourier expansion of the form
\begin{equation}
 F(Z) = \sum_{\substack{S \in M_{2r}^{\sym}(\Q) \\ S \ge 0}}P\left((2 \pi Y)^{-1} ; S\right)e^{2\pi i\,{\rm tr}(SZ)},\qquad Z=X+iY.
\end{equation}
Then $F$ belongs to $N_k(\Gamma, R)$ if and only if the coefficients of all the polynomials $P(X, S)$ lie in $R$. We define $M_k(\Gamma, R) = N_k(\Gamma, R) \cap M_k(\Gamma)$.

We fix an embedding of $\overline{\Q}$ in $\C$. Given a prime $p$, we let $\mathcal{O}_p \subset \overline{\Q}$ denote the subring  consisting of elements whose $\lambda$-adic valuation is non-negative for any prime $\lambda$ lying above $p$. Equivalently, an element $x \in \overline{\Q} \subset \C$ belongs to $\mathcal{O}_p$ if $x$ lies in the ring of integers of  $\overline{\Q_p}$ for every choice of isomorphism $\overline{\Q_p} \simeq \C$. We say that $x \in \overline{\Q}$ is $p$-integral if $x$ belongs to $\mathcal{O}_p$.
 \section{Definition of the Eisenstein series}
\subsection{Preliminaries}\label{basicnotdefsec}
Let $\A$ denote the ring of adeles of $\Q$. Let $P_{2n}$ be the Siegel parabolic subgroup of $G_{2n}$, consisting of matrices whose lower left $n\times n$-block is zero. Let $N_{2n}$ denote the unipotent radical of $P_{2n}$. Let $\delta_{P_{2n}}$ be the modulus character of $P_{2n}(\A)$. It is given by
\begin{equation}\label{modulus-char}
 \delta_{P_{2n}}(\mat{A}{X}{}{v\,^t\!A^{-1}}) = |v^{-\frac{n(n+1)}2}\det(A)^{n+1}|, \qquad \text{ where } A \in \GL_n(\A),\: v \in \GL_1(\A),
\end{equation}
and $|\cdot|$ denotes the global absolute value on $\A$, normalized in the standard way, so that $|x|=1$ for $x\in\Q^\times$.

Fix the following embedding of $H_{2a,2b} := \{(g, g') \in G_{2a} \times G_{2b}: \mu_a(g) = \mu_b(g') \}$ into $G_{2a+2b}$:
\begin{equation}\label{embedding-defn}
H_{2a,2b} \ni (\mat{A_1}{B_1}{C_1}{D_1}, \mat{A_2}{B_2}{C_2}{D_2}) \longmapsto \left[\begin{smallmatrix}A_1&&-B_1\\&A_2&&B_2\\-C_1&&D_1\\&C_2&&D_2\end{smallmatrix}\right] \in \GSp_{2a+2b}.
\end{equation}
We will also let $H_{2a,2b}$ denote its image in $G_{2a+2b}$, and think of $(g,g')\in H_{2a,2b}$ as an element of $G_{2a+2b}$.

For $p=\mat{A}{*}{}{v\,^t\!A^{-1}}$ with $v\in\GL_1(\A)$ and $A\in\GL_{2n}(\A)$, note that $\delta_{P_{4n}}(p)=|v^{-n}\det(A)|^{2n+1}$. Let $\chi$ be a finite order character of $\Q^\times \backslash \A^\times$. We define a character on $P_{4n}(\A)$, also denoted by $\chi$, by $\chi(p) = \chi(v^{-n}\det(A))$.

For a complex number $s$, let
\begin{equation}\label{ind-repn}
 I(\chi,s) = {\rm Ind}^{G_{4n}(\A)}_{P_{4n}(\A)} ( \chi \delta_{P_{4n}}^s).
\end{equation}
Thus, any $f(\,\cdot\,, s) \in I(\chi,s)$ is a smooth complex-valued function satisfying
\begin{equation}\label{ind-repn-fctn}
 f(pg, s)  = \chi(p) \delta_{P_{4n}}(p)^{s + \frac12} f(g,s)
\end{equation}
for all $p \in P_{4n}(\A)$ and $g \in G_{4n}(\A)$. Note that these functions are invariant under the center of $G_{4n}(\A)$. Let $I(\chi_v,s)$  be the analogously defined local representation at a place $v$ of $\Q$.
We have $I(\chi,s)\cong\otimes I(\chi_v,s)$ in a natural way.

Let $f(\cdot,s) \in I(\chi, s)$ be a section whose restriction to the standard maximal compact subgroup  of $G_{4n}(\A)$ is independent of $s$. Consider the Eisenstein series on $G_{4n}(\A)$ which, for $\Re(s)>\frac12$, is given by the absolutely convergent series
\begin{equation}\label{Eis-ser-defn}
 E(g,s,f) = \sum\limits_{\gamma \in P_{4n}(\Q) \backslash G_{4n}(\Q)} f(\gamma g, s),
\end{equation}
and defined outside this region by analytic continuation \cite[Appendix II]{langlands}.
\subsection{The choice of section}\label{section-choice}
For the rest of this paper, let $N$ be an integer and let $S$ be the set of primes dividing $N$. Write $N = \prod_{p \in S} p^{m_p}$, so that  $m_p>0$ for all $p \in S$. Let $\chi=\otimes\chi_p$ be a finite order character of $\Q^\times \backslash \A^\times$ whose conductor divides $N$. In other words, $S$ contains all primes $p$ where $\chi_p$ is ramified, and for each $p \in S$,  $\chi_p |_{(1+p^{m_p}\Z_p)} = 1$. Let $k$ be a positive integer  such that $\chi_\infty = {\rm sgn}^k$.

For $0 \leq r \leq n$, define a matrix of size $4n\times4n$ by
\begin{equation}\label{Qrdefeq}
 Q_r = \left[\begin{smallmatrix}I_n &0&0&0 \\ 0& I'_{n-r}&0&\widetilde{I_r}\\ 0&0&I_n&\widetilde{I_r}\\ \widetilde{I_r}&-\widetilde{I_r}&0&I'_{n-r} \end{smallmatrix}\right],
\end{equation}
where $\widetilde{I_r}= \mat{0_{n-r}}{0}{0}{I_r}$ and $I'_{n-r}=I_n-\tilde I_r=\mat{I_{n-r}}{0}{0}{0}$. Given $f\in I(\chi,s)$, it follows from iii) of Lemma 2.2 of \cite{PSS17} that, for all $h_1, h_2 \in G_{2n}(\A)$ with the same multiplier,
\begin{equation}\label{Qninveq1}
 f(Q_n \cdot (gh_1, gh_2), s)=f(Q_n \cdot (h_1, h_2), s)\qquad\text{for }g\in G_{2n}(\A).
\end{equation}
Choose local sections  as follows.
\begin{enumerate}
 \item Let $p \notin S$ be a finite prime. We choose $f_p\in I(\chi_p,s)$ to be the unique normalized unramified vector, i.e., $f_p:\:G_{4n}(\Q_p) \times \C\rightarrow\C$ is given by
  \begin{equation}\label{sphericalsectioneq}
   f_p(Pk, s)=\chi_p(P)\delta_{P_{4n}}(P)^{s + \frac12}
  \end{equation}
  for $P \in P_{4n}(\Q_p)$ and $k\in G_{4n}(\Z_p)$.
 \item Let $p \in S$. Define $K_p(m_p) := \{ g \in \Sp_{4n}(\Z_p) : \ g \equiv I_{4n}\bmod p^{m_p}\}$. Let $f_p(g,s)$ be the unique function on $G_{4n}(\Q_p) \times \C$ such that
 \begin{equation}\label{badplacecond1}
  f_p(Pk, s)  = \chi_p(P) \delta_{P_{4n}}(P)^{s + \frac12}\quad\text{ for all }P\in P_{4n}(\Q_p)\text{ and }k\in K_p(m_p)
 \end{equation}
 and
 \begin{equation}\label{badplacecond2}
  f_p(g, s) = 0\quad\text{ if }g \notin  P_{4n}(\Q_p)K_p(m_p).
 \end{equation}
  It is easy to see that $f_p$ is well-defined. Evidently, $f_p\in I(\chi_p,s)$.
 \item Let $v = \infty$. Let $K^{(2n)} \simeq U(2n)$  be the standard maximal compact subgroup of $\Sp_{4n}(\R)$.

 Let
  \begin{equation}\label{degprincserhollemmaeq3c}
   f_\infty(\mat{A}{*}{}{u\,^t\!A^{-1}}g,s)= \sgn^k(\det(A))\,\sgn^{nk}(u)|u^{-n}\det(A)|^{(2n+1)(s+\frac12)}\,J(g,I)^{-k}
  \end{equation}
  for  $A\in\GL_{2n}(\R)$, $u\in\R^\times$ and $g\in K^{(2n)}$.
\end{enumerate}
Define
\begin{equation}\label{globalsectioneq}
 f(g, s) = \prod_v f_v(g_v, s).
\end{equation}
Let $Q$ denote the element $Q_n$ embedded diagonally in $\prod_{p<\infty}\Sp_{4n}(\Z_p)$, and for any $\tau \in \hat{\Z}^\times = \prod_{p<\infty}\Z_p^\times$, let
\begin{equation}\label{Qtaueq}
 Q_\tau = \mat{\tau I_{2n}}{}{}{I_{2n}}Q \mat{\tau^{-1} I_{2n}}{}{}{I_{2n}}.
\end{equation}
For $f$ as in \eqref{globalsectioneq} and any $h \in \Sp_{4n}(\hat{\Z})=\prod_{p<\infty}\Sp_{4n}(\Z_p)$, define $f^{(h)}(g,s) = f(gh^{-1}, s)$. We can now define the Eisenstein series $E(g, s, f^{(Q_\tau)})$ to be the Eisenstein series defined in (\ref{Eis-ser-defn}) with $f$, $Q_\tau$  as above.
\subsection{The Eisenstein series \texorpdfstring{$E_{k,N}^{\chi}(Z,s; h)$}{}}\label{s:classicaldef}
For $Z \in \H_{2n}$ and $h \in \Sp_{4n}(\hat{\Z})$, we define
\begin{equation}\label{classical-Eis-ser-1}
 E_{k,N}^{\chi}(Z,s; h):=J(g, I)^{k} E\Big(g,\frac {k+2s}{2n+1}-\frac 12, f^{(h)}\Big),
\end{equation}
where $g$ is any element of $\Sp_{4n}(\R)$ with $g(I)=Z$. (For a proof that this definition coincides with the earlier one given by \eqref{eisensteinserieseq}, see Section 6.2 of \cite{PSS17}.) We will be interested in the function $E_{k,N}^{\chi}(Z,-m_0; h)$, where $m_0 \ge 0$ is a non-negative integer. The following result is a consequence of \cite[Prop. 6.6, 6.8]{PSS17}.
\begin{proposition}\label{propnearholo}
 Suppose that $k\ge n+1$ and let $h \in \prod_{p<\infty}\Sp_{4n}(\Z_p)$. If $k=n+1$, assume further that $\chi^2 \neq 1$. Let $0 \le m_0 \le \frac{k}2-\frac{n+1}2$ be an integer, and exclude the case $m_0 = \frac{k}2-\frac{n+1}2$, $\chi^2 = 1$. Then
 $$
  \pi^{-2 m_0 n}E_{k,N}^{\chi}(Z, -m_0;h) \in N_k(\Gamma_{4n}(N), \Q_{\rm{ab}}),
 $$
 where $\Q_{\rm{ab}}$ is the maximal abelian extension of $\Q$.
\end{proposition}

%
%

We note, however, that the series defining $E_{k,N}^{\chi}(Z,-m_0 ; h)$  converges absolutely  only in the smaller range  $m_0 \le \frac{k}{2} - n -1$.
\section{Cuspidality of the pullback}\label{s:cuspidal}
The results of this section will complete the proof of assertion 1 of Theorem \ref{main-Eis-arith-thm}.
\subsection{Formulation of the main result}


Let the  Eisenstein series $E(g, s, f^{(Q_\tau)})$ be as defined in the previous section. In this section we will prove the following theorem.

\begin{theorem}\label{Eis-cusp-thm}
Let $0 \leq m_0 \leq k/2-n-1$. Assume that $N>1$. Then the restriction of $E( - , \frac {k-2m_0}{2n+1}-\frac 12, f^{(Q_\tau)})$ to $H_{2n,2n}(\A)$ gives a cuspidal automorphic form on $H_{2n,2n}(\A)$.
\end{theorem}

We clarify what cuspidality means. Let $s=\frac {k-2m_0}{2n+1}-\frac 12$. The above theorem states that
$$
 \int\limits_{R(\Q) \backslash R(\A)} E((ug_1, g_2), s, f^{(Q_\tau)})\,du=\int\limits_{R(\Q) \backslash R(\A)} E((g_1, ug_2), s, f^{(Q_\tau)})\,du=0
$$
for each $(g_1,g_2)\in H_{2n,2n}(\A)$, and the unipotent radical $R$ of any maximal parabolic subgroup of $\Sp_{2n}$.

A nearly-holomorphic modular form is cuspidal in the classical sense (i.e., its Fourier expansion at each cusp is supported on positive definite matrices), if and only if its adelization is cuspidal in the sense that its integral over the unipotent radical of any maximal parabolic subgroup vanishes. For details of this argument, see, e.g., the proof of Proposition 4.5 of \cite{PSS14}.

Therefore Theorem \ref{Eis-cusp-thm} is equivalent to the following statement:  For each $m_0$ with $0 \le m_0 \le \frac{k}{2} - n -1$,
\begin{equation}\label{Eis-eqn-1}
 E_{k,N}^{\chi}\left(\mat{Z_1}{}{}{Z_2},-m_0; h\right) \in N_k(\Gamma_{2n}(N))^\circ \otimes  N_k(\Gamma_{2n}(N))^\circ,
\end{equation}
which is precisely the assertion of part 1 of Theorem \ref{main-Eis-arith-thm}.
\subsection{Unwinding of the Eisenstein series}
For the rest of Section \ref{s:cuspidal}, we will assume that $N>1$, or equivalently, that $S$ is non-empty. By Proposition 2.1 of \cite{PSS17}, we have the double coset decomposition
\begin{equation}\label{PSS17Prop21}
  G_{4n}(\Q) = \bigsqcup_{r=0}^n  P_{4n}(\Q)Q_r H_{2n,2n}(\Q),
\end{equation}
where the $Q_r$ are defined in (\ref{Qrdefeq}). By \eqref{Eis-ser-defn} and \eqref{PSS17Prop21},
\begin{equation}\label{Eis-unwind}
 E(g, s, f^{(Q_\tau)}) = \sum\limits_{r=0}^n\:\sum\limits_{\gamma \in \Delta_r \backslash H_{2n,2n}(\Q)} f^{(Q_\tau)}(Q_r \gamma g, s),
\end{equation}
where $\Delta_r := H_{2n,2n}(\Q) \cap Q_r^{-1} P_{4n}(\Q) Q_r$.

\begin{lemma}\label{support-f-ram}
 Let $p$ be a prime in $S$. Let the integer $m_p>0$ and the group $K_p(m_p)$ be as defined in Section \ref{section-choice}. Let $g_1, g_2 \in G_{2n}(\Q_p)$ with the same multiplier. Then, for $0 \leq r < n$,
 $$
  Q_r (g_1, g_2) \not\in P_{4n}(\Q_p) Q_n K_p(m_p).
 $$
\end{lemma}
\begin{proof}
Suppose that $Q_r (g_1, g_2) \in P_{4n}(\Q_p) Q_n K_p(m_p)$. Then there is a $p' \in P_{4n}(\Q_p)$ such that $p' Q_r (g_1, g_2) Q_n^{-1} \in K_p(m_p)$. Here we have used that $K_p(m_p)$ is normalized by $Q_n$. Suppose $g_i = \mat{A_i}{B_i}{C_i}{D_i}$ for $i=1,2$. Then
$$
 Q_r (g_1, g_2) Q_n^{-1} = \begin{bsmallmatrix}\ast&\ast&\ast&\ast\\\ast&\ast&\ast&\ast\\\ast&\tilde{I}_rD_2-D_1&D_1&-\tilde{I}_rC_2\\ \ast & I_{n-r}'D_2+\tilde{I}_r(B_1-B_2)&-\tilde{I}_rB_1&\tilde{I}_rA_2- I_{n-r}'C_2\end{bsmallmatrix}
$$
where  $\widetilde{I_r}= \mat{0_{n-r}}{0}{0}{I_r}$ and $I'_{n-r}=I_n-\tilde I_r=\mat{I_{n-r}}{0}{0}{0}$. Write $p' \in P_{4n}$ with lower right entry being $\mat{h_1}{h_2}{h_3}{h_4}$. Then we see from the six entries in the last two rows and last three columns of $p' Q_r (g_1, g_2) Q_n^{-1}$ that
\begin{align*}
& h_1 (\tilde{I}_rD_2-D_1) + h_2( I_{n-r}'D_2+\tilde{I}_r(B_1-B_2)) \in M_n(p^m), \qquad h_1 D_1 - h_2 \tilde{I}_rB_1 \in I_n + M_n(p^m), \\
& -h_1 \tilde{I}_rC_2 + h_2 (\tilde{I}_rA_2- I_{n-r}'C_2) \in M_n(p^m), \\
&h_3 (\tilde{I}_rD_2-D_1) + h_4 (I_{n-r}'D_2+\tilde{I}_r(B_1-B_2)) \in M_n(p^m), \qquad h_3 D_1 - h_4 \tilde{I}_rB_1 \in M_n(p^m), \\
&  -h_3 \tilde{I}_rC_2 + h_4 (\tilde{I}_rA_2- I_{n-r}'C_2) \in I_n + M_n(p^m).
\end{align*}
We can put this in matrix form as follows,
$$
 \mat{-h_2 \tilde{I}_r}{h_1  \tilde{I}_r + h_2 I_{n-r}'}{-h_4 \tilde{I}_r}{h_3  \tilde{I}_r + h_4 I_{n-r}'} \mat{B_2}{A_2}{D_2}{C_2}\equiv 1 \pmod{p^m}.
$$
The first $n-r$ columns of the leftmost matrix are zero, and therefore, since $r<n$, the left hand side has determinant zero. This contradiction proves the lemma.
\end{proof}

It is easy to see that $P_{4n}(\Q_p) Q_n =  P_{4n}(\Q_p)\mat{\tau I_{2n}}{}{}{I_{2n}}Q_n \mat{\tau^{-1} I_{2n}}{}{}{I_{2n}}$. Together with \eqref{badplacecond2}, this implies that, for $p\in S$, the support of the local section $f_p^{(Q_\tau)}(\cdot,s)$ is $P_{4n}(\Q_p) Q_n K_p(m_p)$. Using Lemma \ref{support-f-ram}, it follows that
$$
 E((g_1, g_2), s, f^{(Q_\tau)}) =  \sum\limits_{\gamma \in \Delta_n \backslash H_{2n,2n}(\Q)} f^{(Q_\tau)}(Q_n \gamma (g_1, g_2), s).
$$
By Proposition 2.3 of \cite{PSS17}, a set of representatives of $\Delta_n \backslash H_{2n,2n}(\Q)$ is $\{(x,1) : x \in \Sp_{2n}(\Q)\}$. Hence
\begin{equation}\label{Eis-n}
 E((g_1, g_2), s, f^{(Q_\tau)}) =  \sum\limits_{x \in \Sp_{2n}(\Q)} f^{(Q_\tau)}(Q_n (xg_1, g_2), s).
\end{equation}
\subsection{The proof of Theorem \ref{Eis-cusp-thm}}


Let $\psi$ be the additive character of $\Q \backslash \A$ which is $x \mapsto e^{2 \pi i x}$ for $x \in \R$ and is trivial on $\Z_p$ for every finite prime $p$.  For a symmetric matrix $S$ in $M_n(\Q)$, we obtain a character of the unipotent radical of the Siegel parabolic by setting
$$
 U(\A) \ni u(X) = \mat{1}{X}{}{1} \mapsto \theta_S(u(X)) := \psi({\rm tr}(SX)).
$$
Note that we earlier used $S$ to denote the set of primes dividing $N$. From the context, the meaning of $S$ should be clear.
\begin{lemma}\label{cuspidalitycriterion}
 Let $\Phi$ be an automorphic form on $\Sp_{2n}(\A)$. Assume that
 \begin{equation}\label{cuspidalitycriterioneq1}
  \int\limits_{U(\Q)\backslash U(\A)}\Phi(ug)\theta_S^{-1}(u)\,du=0
 \end{equation}
 for all $g\in\Sp_{2n}(\A)$ and all $S\in M^{\rm sym}_n(\Q)$ for which the first row and column are zero. Then $\Phi$ is a cusp form.
\end{lemma}
\begin{proof}
We have
\begin{equation}\label{cuspidalitycriterioneq2}
 M^{\rm sym}_n=M'\oplus M^{\rm sym}_{n-1},
\end{equation}
where we think of $M^{\rm sym}_{n-1}$ embedded into $M^{\rm sym}_n$ as matrices whose first row and column are zero, and where
\begin{equation}\label{cuspidalitycriterioneq3}
 M'=\begin{bsmallmatrix} *&*&\ldots&*\\ *&0&\ldots&0\\\vdots&\vdots&&\vdots\\ *&0&\ldots&0\end{bsmallmatrix}\subset M^{\rm sym}_n.
\end{equation}
Hence $M'$ is a space of dimension $n$. Let $S\in M^{\rm sym}_{n-1}(\Q)$, considered as an element of $M^{\rm sym}_n(\Q)$ via \eqref{cuspidalitycriterioneq2}.
Since $\theta_S^{-1}(X) = - \psi({\rm tr}(SX))$, by hypothesis, for any $g\in\Sp_{2n}(\A)$,
\begin{equation}\label{cuspidalitycriterioneq4}
 \int\limits_{M^{\rm sym}_n(\Q)\backslash M^{\rm sym}_n(\A)}\Phi(\mat{1}{X}{}{1}g)\psi({\rm tr}(SX))\,dX=0.
\end{equation}
Using \eqref{cuspidalitycriterioneq2}, it follows that
\begin{align*}
 0&=\int\limits_{M^{\rm sym}_{n-1}(\Q)\backslash M^{\rm sym}_{n-1}(\A)}\:\int\limits_{M'(\Q)\backslash M'(\A)}\Phi(\mat{1}{X+Y}{}{1}g)\psi({\rm tr}(S(X+Y)))\,dY\,dX\\
 &=\int\limits_{M^{\rm sym}_{n-1}(\Q)\backslash M^{\rm sym}_{n-1}(\A)}\bigg(\int\limits_{M'(\Q)\backslash M'(\A)}\Phi(\mat{1}{X}{}{1}\mat{1}{Y}{}{1}g)\,dY\bigg)\psi({\rm tr}(SX))\,dX.
\end{align*}
Hence the inner integral is zero as a function of $X$, and in particular
\begin{equation}\label{cuspidalitycriterioneq5}
 \int\limits_{M'(\Q)\backslash M'(\A)}\Phi(\mat{1}{Y}{}{1}g)\,dY=0.
\end{equation}
Since $M'$ is the intersection of the unipotent radicals of all standard maximal parabolics, it follows that $\Phi$ is cuspidal.
\end{proof}

We are now ready to prove Theorem \ref{Eis-cusp-thm}. Fix $(g_1,g_2)\in H_{2n,2n}(\A)$. By Lemma \ref{cuspidalitycriterion}, and by symmetry between the first and second variables, it suffices to prove that
$$
 \int\limits_{U(\Q) \backslash U(\A)} E((ug_1, g_2), s, f^{(Q_\tau)}) \theta_S^{-1}(u)\,du=0
$$
for all $S\in M^{\rm sym}_n(\Q)$ whose first row and column are zero.
%
%
%
%
Recall that we are choosing values of $s$ such that $(2n+1)(s+1/2) = k - 2m_0$ for $0 \leq m_0 \leq k/2-n-1$.
 From (\ref{Eis-n}), we see that
\begin{align*}
 E((g_1, g_2), s, f^{(Q_\tau)}) &= \sum\limits_{x \in \Sp_{2n}(\Q)/U(\Q)}\:\sum\limits_{v \in U(\Q)} f^{(Q_\tau)}(Q_n (xvg_1, g_2), s) \\
&= \sum\limits_{x \in \Sp_{2n}(\Q)/U(\Q)}\:\sum\limits_{v \in U(\Q)} f^{(Q_\tau)}(Q_n (vg_1, x^{-1}g_2), s).
\end{align*}
Here, we have used (\ref{Qninveq1}). Hence it is enough to show that
$$
 \int\limits_{U(\Q)\backslash U(\A)} \theta_S^{-1}(u) \sum\limits_{v \in U(\Q)} f^{(Q_\tau)}(Q_n (vug_1, g_2), s)\,du = \int\limits_{U(\A)} \theta_S^{-1}(u)  f^{(Q_\tau)}(Q_n (ug_1, g_2), s)\,du
$$
is zero. The above integral is Eulerian. We will show that any of the archimedean components
$$
 \int\limits_{U(\R)} \theta_S^{-1}(u)f_\infty(Q_n (ug_1, g_2), s)\,du
$$
is already zero. By the definition of $f_\infty$ given in (\ref{degprincserhollemmaeq3c}), it is enough to consider $g_1 = Q(v,y)$, $g_2 = u(X_0)Q(v, y_0) \in G_{2n}(\R)$. Here $Q(v,y) = \mat{y}{}{}{v\,^ty^{-1}}$ for $v > 0$ and $y \in \GL_n(\R)^+$. Recall, for $(g_1, g_2) \in H_{2n,2n}$ it is necessary that $\mu_n(g_1) = \mu_n(g_2)$. Once again using (\ref{Qninveq1}),
\begin{align*}
 & \int\limits_{ U(\R)} \theta_S^{-1}(u)  f_\infty(Q_n (uQ(v,y), u(X_0)Q(v,y_0)), s)\,du\\
 &\qquad=\int\limits_{M_{n}^{\sym}(\R)} \theta_S^{-1}(u(X))  f_\infty(Q_n (u(X)Q(v,y), u(X_0)Q(v,y_0)), s)\,dX \\
 &\qquad = \int\limits_{M_{n}^{\sym}(\R)} \theta_S^{-1}(u(X))  f_\infty(Q_n (u(y_0^{-1}(X-X_0)v\,^ty_0^{-1})Q(1,y_0^{-1}y), 1), s)\,dX\\
 & \qquad = \theta_S^{-1}(u(X_0))v^{-\frac{n(n+1)}2}\int\limits_{M_{n}^{\sym}(\R)} \theta_{v^{-1}S}^{-1}(u(X))  f_\infty(Q_n (u(y_0^{-1}X\,^ty_0^{-1})Q(1,y_0^{-1}y), 1), s)\,dX.
\end{align*}
The factor $v^{-\frac{n(n+1)}2}$ comes from the variable transformation $X\mapsto v^{-1}X$.
We will show that, up to a factor, we can reduce to evaluating the $f_\infty$ term at a block lower triangular matrix. Let us abbreviate $B = y_0^{-1}X\,^ty_0^{-1}$ and $A = y_0^{-1}y$. Then
\begin{align}\label{finftycalceq1}
 f_\infty(Q_n(u(B) Q(1, A), 1), s)&= f_\infty(\begin{bsmallmatrix}I_n\\&I_n\\&I_n&I_n\\I_n&&&I_n\end{bsmallmatrix} \begin{bsmallmatrix}I_n\\&&&I_n\\&&I_n\\&-I_n\end{bsmallmatrix} (u(B) Q(1, A), 1), s)\nonumber\\
 &=f_\infty(\begin{bsmallmatrix}I_n\\&I_n\\&I_n&I_n\\I_n&&&I_n\end{bsmallmatrix}  \begin{bsmallmatrix}I_n&&-B\\&I_n\\&&I_n\\&&&I_n\end{bsmallmatrix}(Q(1, A), 1), s)\nonumber\\
 &=(-i)^{-nk}f_\infty(\begin{bsmallmatrix}I_n&B&-B\\&I_n\\&&I_n\\&B&-B&I_n\end{bsmallmatrix} \begin{bsmallmatrix}I_n\\&I_n\\&I_n&I_n\\I_n&&&I_n\end{bsmallmatrix} (Q(1, A), 1), s)\nonumber\\
 &=(-i)^{-nk}f_\infty(\begin{bsmallmatrix}I_n\\&I_n\\&&I_n\\&B&&I_n\end{bsmallmatrix} \begin{bsmallmatrix}I_n\\&I_n\\&I_n&I_n\\I_n&&&I_n\end{bsmallmatrix} \begin{bsmallmatrix}A\\&I_n\\&&\,^t\!A^{-1}\\&&&I_n\end{bsmallmatrix}, s)\nonumber\\
 &=(-i)^{-nk}f_\infty(\begin{bsmallmatrix}A\\&I_n\\&&\,^t\!A^{-1}\\&&&I_n\end{bsmallmatrix} \begin{bsmallmatrix}I_n\\&I_n\\&\,^t\!A&I_n\\A&B&&I_n\end{bsmallmatrix}, s)\nonumber\\
 &=(-i)^{-nk}{\rm sgn}^k(\det(A)) |\det(A)|^{k-2m_0}f_\infty(\begin{bsmallmatrix}I_n\\&I_n\\&{}^tA&I_n\\A&B&&I_n\end{bsmallmatrix}, s).
\end{align}

\begin{lemma}\label{Iwasawa-lemma}
 Let $C \in M_{2n}(\R)$ be a symmetric matrix such that  $I_{2n}+C^2$ is invertible. Then
 \begin{equation}\label{Iwasawa-eqn}
  \mat{I_{2n}}{}{C}{I_{2n}} = \mat{I_{2n}}{U}{}{I_{2n}} \mat{Y}{}{}{{}^tY^{-1}} g,
 \end{equation}
 where
 $$
  U = C(I_{2n}+C^2)^{-1}, \qquad Y{}^tY = (I_{2n}+C^2)^{-1},
 $$
 and $g \in K^{(2n)}\simeq U(2n)$ is such that $J(g, iI_{2n}) = \det(Y) \det(iC+I_{2n})$.
\end{lemma}

\begin{proof}
Act with both sides of \eqref{Iwasawa-eqn} on $iI_{2n}$ to get $iI_{2n} (iC+I_{2n})^{-1} = U + i Y{}^tY$. Hence we get $iI_{2n} = (U + i Y{}^tY)(iC+I_{2n}) = U - Y{}^tYC+i(UC+Y\,^tY)$. Comparing the real and imaginary parts of both sides we get the values of $U$ and $Y$ in terms of $C$. Applying the $J$ function to both sides of (\ref{Iwasawa-eqn}), we get $J(\mat{I_{2n}}{}{C}{I_{2n}}, iI_{2n}) = J(\mat{Y}{}{}{{}^tY^{-1}}, iI_{2n}) J(g, iI_{2n})$. This concludes the proof.
\end{proof}

We now put $C = \mat{}{{}^t\!A}{A}{B}$ and use Lemma \ref{Iwasawa-lemma} to evaluate $f_\infty(\ldots)$ in \eqref{finftycalceq1}. We have
\begin{align*}
f_\infty(\mat{I_{2n}}{}{C}{I_{2n}}, s) &= f_\infty(\mat{I_{2n}}{U}{}{I_{2n}} \mat{Y}{}{}{{}^tY^{-1}} g, s)\\
 &\stackrel{\eqref{degprincserhollemmaeq3c}}{=} {\rm sgn}^k(\det(Y)) |\det(Y)|^{k-2m_0} J(g, iI_{2n})^{-k} \\
 &=  {\rm sgn}^k(\det(Y)) |\det(Y)|^{k-2m_0} \det(Y)^{-k} \det(I_{2n}+iC)^{-k} \\
 &= |\det(Y)|^{-2m_0} \det(I_{2n}+iC)^{-k} \\
 &= \det(I_{2n}+iC)^{m_0-k} \det(I_{2n}-iC)^{m_0}.
\end{align*}
Set $Z = -i(y\,^t\!y + y_0{}^t\!y_0)$. Computing both sides of the equality below, we get
$$
 \det(I_{2n}+iC)= \det(y_0)^{-2}i^n \det(X+Z).
$$
(Observe that $X+Z$ is invertible, since $X+\bar Z$ is an element of the Siegel upper half space $\H_n$.) Putting all of this together, we get
$$
 f_\infty(Q_n(u(B) Q(1, A), 1), s) = \det(y_0y)^{k-2m_0} \det(X+Z)^{m_0-k} \det(X+\bar{Z})^{m_0}.
$$
Hence, we want to show that the integral
\begin{equation}\label{Int0-1}
 \int\limits_{M_{n}^{\sym}(\R)} \theta_S^{-1}(u(X))  \det(X+Z)^{m_0-k} \det(X+\bar{Z})^{m_0}\,dX
\end{equation}
is zero for any $S\in M^{\rm sym}_n(\R)$ for which the first row and column are zero.

Let $x$ denote the $(1,1)$ matrix entry of the variable $X$. By assumption on $S$, the quantity $\theta_S^{-1}(u(X))$ does not depend on $x$. Let $M_{ij}$ be the submatrix of $X+Z$ obtained by eliminating the $i$-th row and $j$-th column. Then, expanding along the first row, we get
\begin{align*}
 \det(X+Z)&=\det(\begin{bsmallmatrix}x+z_{11}&x_{12}+z_{12}&\ldots&x_{1n}+z_{1n}\\x_{21}+z_{21}&x_{22}+z_{22}&&\vdots\\\vdots&&\ddots&\vdots\\x_{n1}+z_{n1}&\ldots&\ldots&x_{nn}+z_{nn}\end{bsmallmatrix})\\
 &=(x+z_{11})\det(M_{11})+\sum_{j=2}^n(-1)^{j+1}(x_{1j}+z_{1j})\det(M_{1j}).
\end{align*}
Observe that $\bar M_{11}$ is an element of the Siegel upper half space of degree $n-1$, and hence invertible. We can therefore write
$$
 \det(X+Z)=\det(M_{11})\bigg(x+z_{11}+\det(M_{11})^{-1}\sum_{j=2}^n(-1)^{j+1}(x_{1j}+z_{1j})\det(M_{1j})\bigg).
$$
Hence $\det(X+Z)=f(X)(x+g(X))$ with functions $f$ and $g$ that do not depend on $x$. We also need that $g(X)$ is not real; this follows from Lemma \ref{Hn11lemma} below. This discussion implies that, in \eqref{Int0-1}, there is an inner integral with respect to the variable $x$ of the form
$$
 \int\limits_{-\infty}^\infty (x+z)^{m_0-k}(x+\bar{z})^{m_0}\,dx,\qquad z\notin\R.
$$
Set $I(\ell,m)=\int_{-\infty}^\infty (x+z)^{m-\ell}(x+\bar{z})^m\,dx$ for integers $\ell,m$ with $2m-\ell<-1$. Using integration by parts, one can check that
$$
 I(\ell, m) = \frac m{\ell-m -1} I(\ell-2,m-1).
$$
Applying the above relation $m$ times, we get
$$
 I(\ell,m) = \frac{m! (\ell-2m-1)!}{(\ell-m-1)!} I(\ell-2m,0).
$$
Since $\ell-2m > 1$ and $z\notin\R$, we can easily see that $I(\ell-2m,0) = 0$. Hence $I(\ell,m)=0$ whenever $2m-\ell<-1$. Since our assumption is that $0 \leq m_0 \leq k/2-n-1$, it follows that $I(k,m_0)=0$. This concludes the proof of Theorem \ref{Eis-cusp-thm}.\qed

\begin{lemma}\label{Hn11lemma}
 Let $M\in\H_n$, and let $M_{11}$ be the submatrix of $M$ obtained by eliminating the first row and the first column. Then
 $$
  \frac{\det(M)}{\det(M_{11})}\notin\R.
 $$
\end{lemma}
\begin{proof}
The reciprocal $\frac{\det(M_{11})}{\det(M)}$ is the $(1,1)$-coefficient of $M^{-1}$. Since $-M^{-1}=\mat{}{I_n}{-I_n}{}M\in\H_n$, this coefficient is not real.
\end{proof}
\section{The integrality result at \texorpdfstring{$m_0=0$}{}}\label{s:initial}
\subsection{An initial integrality result}
Recall the definition of the Eisenstein series  $E_{k,N}^{\chi}(Z,-m_0; h)$ from Section \ref{s:classicaldef}. Let $\iota = \mat{0}{I_{2n}}{\!\!-I_{2n}}{0}$. We will first prove the following result.
\begin{proposition}\label{integrality-eis-series-thm}
 Let $k,n,N$ be positive integers with $k \ge n+1$ and $N>1$. If $k=n+1$, assume further that $\chi^2 \neq 1$. Let $p$ be any prime such that $p \nmid 2N$ and $p \geq 2k$. Then
 \begin{equation}\label{integrality-eis-series-eqn}
  \pi^{n + n^2-(2n+1)k} \Lambda^N\Big(\frac k2\Big) E_{k,N}^{\chi}(Z,0; \iota) \in M_k(\Gamma_{4n}(N),\mathcal{O}_p).
 \end{equation}
\end{proposition}

\begin{remark}\label{rem:brown}The above result is closely related to Theorem 4.4 of Brown \cite{Brown2007}, which states a similar result for the Eisenstein series at the point $m_0=(2n+1)/2-k$ instead of $m_0=0$. (The proof of the above theorem as it appears in \cite{Brown2007} has a gap which has been recently corrected by Brown  in  a note available at his webpage\footnote{Retrieved in February 2021 from \url{http://jim-brown.oxycreates.org/research.html}.}.) The points $(2n+1)/2-k$ and $0$ are related by the functional equation for the Eisenstein series. So in principle, one can try to deduce Proposition \ref{integrality-eis-series-thm} from Brown's result.  However, the functional equation will involve factors coming from the local intertwining operators at
the primes dividing the level and the $p$-integrality of these factors will need to be proven. We give an independent proof of Proposition \ref{integrality-eis-series-thm} that does not rely on Brown's result.

\end{remark}
\subsection{Fourier expansion of Eisenstein series}\label{s:brownproof}

Let the setup be as in the statement of Proposition \ref{integrality-eis-series-thm}. We begin by noting that
\begin{equation}\label{step1eq1}
 E_{k,N}^{\chi}(Z,s; 1) =  E(Z,s+k/2;k,\chi, N),
\end{equation}
where the Eisenstein series on the right hand side is defined by Shimura in \cite[(16.40)]{shibook2}. (See Remark 6.2 of \cite{PSS17}).
 We also note that
\begin{equation}\label{Eis-iota-reln}
 E_{k,N}^{\chi}(Z,s; \iota) =  E^\ast(Z,s+k/2;k,\chi, N).
\end{equation}
Here,  $E^\ast$ is defined in \cite[(16.33)]{shibook2}.   The Fourier expansion of $E^\ast(Z,s) := E^\ast(Z,s;k,\chi, N)$ is given by \cite[(16.42)]{shibook2} as follows,
\begin{equation}\label{Fourier-exp-Eis-series}
 E^\ast(Z,s) = \sum\limits_{h \in N^{-1} L}c(h, Y, s) e^{2 \pi i {\rm tr}(hX)},\qquad Z=X+iY.
\end{equation}
Here $L$ is the set of symmetric $2n \times 2n$ half integral matrices.

The argument on page 460 of \cite{shi83} implies that $c(h, Y, k/2)$ are non-zero only when $h$ is positive definite (this fact requires $N>1$). So we henceforth assume that $h$ is positive definite. For any Hecke character $\rho$ of $\Q^\times \bs \A^\times$, define
$$L^N(s, \rho) = \prod_{p<\infty, \ p\nmid N}L(s, \rho_p).$$
Furthermore, if $\eta$ is a primitive Dirichlet character, then we take the $L$-function of $\eta$ to be that of the associated  Hecke character (see Section 7.1 of \cite{PSS17}).

Proposition 16.9 of \cite{shibook2}, specialized to our case, gives the following formula for the Fourier coefficients,
\begin{equation}\label{Four-coeff-formula}
 c(h, Y, s) = \det(Y)^{s-k/2} N^{-\frac{2n(2n+1)}2} \alpha_N(h; 2s, \chi) \xi\left(Y, h, s+\frac k2, s-\frac k2\right),
\end{equation}
where, by Proposition 16.10 of \cite{shibook2},
\begin{equation}\label{alphaNformula}
 \alpha_N(h; 2s, \chi) = \Lambda^N(s)^{-1} L^N(2s-n, \chi \rho_h) \prod\limits_{\ell \in {\bf C}} f_{h,  \ell}(\chi(\ell)\ell^{-2s}),
\end{equation}
with
\begin{equation}\label{LambdaNdefeq}
 \Lambda^N(s) := L^N(2s, \chi) \prod\limits_{i=1}^{n}L^N(4s-2i, \chi^2).
\end{equation}
Above, $\rho_h$ is the quadratic character corresponding to the  extension $\Q(\sqrt{(-1)^{n}\det(2h)})/\Q$, ${\bf C}$ consists of the set of primes that divide $\det(2h)$ but do not divide $N$, and $f_{h,  \ell}$ is a polynomial whose coefficients are in $\Z$ and depend only on $h$, $\ell$, and $n$ (to be made more precise in Section~\ref{s:proof41}). Next,
$$
 \xi(Y, h, s, t) := \int\limits_{M_{2n}^{\rm sym}(\R)}e^{-2\pi i {\rm tr}(hX)} \det(X+iY)^{-s} \det(X-iY)^{-t} dX.
$$
We  have
\begin{equation}\label{xi-formula}
\xi(Y, h; k, 0) = (-1)^{-nk}2^k\pi^{2nk-n^2}\det(h)^{k-\frac{2n+1}2} \prod\limits_{j=1}^{n-1}\frac{2^{2k-2j-2}}{(2k-2j-2)!}e^{-2 \pi {\rm tr}(Yh)}.
\end{equation}
This follows from formulas (4.34.K) and (4.35.K) of \cite{shiconfluent} by setting $\alpha = k, \beta = 0, p = 2n, q = r = 0$.
Hence, we see that
$$
 \Lambda^N(k/2) E^\ast(Z,k/2) = N^{-{n(2n+1)}} (-1)^{nk}2^k\pi^{2nk-n^2} \prod\limits_{j=1}^{n-1}\frac{2^{2k-2j-2}}{(2k-2j-2)!} \sum\limits_{\substack{h \in N^{-1}L \\ h > 0}}b(h) e^{2 \pi i {\rm tr}(hZ)},
$$
where
\begin{equation}\label{b(h)-formula}
 b(h) = \det(h)^{k-\frac{2n+1}2} L^N(k-n, \chi \rho_h) \prod\limits_{\ell \in {\bf C}} f_{h,  \ell}(\chi(\ell)\ell^{-k}).
\end{equation}
Proposition \ref{integrality-eis-series-thm} will follow if we can show that $\pi^{n-k}b(h) \in \mathcal O_p$.

Let $N_h$ be the conductor of $\chi \rho_h$ and let $C_h$ be the conductor of $\rho_h$. Since the conductor of $\chi$ divides $N$, and $p \nmid N$, we can write $N_h = u C_h$, where $u \in \Q$ is such that $p$ does not divide the numerator or denominator of $u$. Note that
\begin{equation}\label{twobrackets}
 \pi^{n-k}b(h) = u^{n-k+\frac 12} \left[\Big(\frac{\det(h)}{C_h}\Big)^{k-n-\frac 12} \prod\limits_{\ell \in {\bf C}} f_{h,  \ell}(\chi(\ell)\ell^{-k})\right] \ \left[\pi^{n-k} N_h^{k-n-\frac 12} L^N(k-n, \chi \rho_h)\right].
\end{equation}
Clearly, $u^{n-k+\frac 12} \in \mathcal O_p^\times$. We will show that the two expressions in the square-brackets above are individually in $\mathcal O_p$, which will complete the proof.
\subsection{Integrality of \texorpdfstring{$L$}{}-values of Hecke characters}
We continue with the notations from the previous section. The following lemma takes care of the second bracketed expression in \eqref{twobrackets}.
%

\begin{lemma}\label{L-fn-integral}
 Let the hypotheses be as in Proposition \ref{integrality-eis-series-thm}. Then
 $$
  \pi^{n-k} N_h^{k-n-\frac 12} L^N(k-n, \chi \rho_h) \in \mathcal O_p.
 $$
\end{lemma}
\begin{proof}
Let $\eta_h$ denote the primitive Dirichlet character corresponding to $\chi \rho_h$. We have
$$
 L^N(k-n, \chi \rho_h) =   L(k-n, \chi \rho_h) \prod\limits_{\ell | N} (1-\eta_h(\ell) \ell^{n - k}).
$$
Since we have assumed $p \nmid N$, we see that $L^N(k-n, \chi \rho_h)$ in the statement of the lemma can be replaced by $L(k-n, \chi \rho_h) = L(k-n, \eta_h)$.
%
Note that, by definition of $\chi$ and $\rho_h$, we have $(-1)^{k-n} = \big(\chi \rho_h\big)(-1) = \eta_h(-1) =: (-1)^\epsilon$, with $\epsilon \in \{0, 1\}$. Hence, Corollary 2.10 in Chapter VII of \cite{Neu} implies that
\begin{equation}\label{Dirichlet-L-value}
L(k-n, \eta_h) = (-1)^{1+(k-n-\epsilon)/2} \frac{G(\eta_h)}{2i^\epsilon} \Big(\frac{2 \pi}{N_h}\Big)^{k-n} \frac{B_{k-n, \overline{\eta_h}}}{(k-n)!}.
\end{equation}
Here,  the Gauss sum is defined by
$$
 G(\eta_h) = \sum\limits_{\nu = 1}^{N_h-1} \eta_h(\nu) e^{2\pi i \nu/{N_h}},
$$
and the generalized Bernoulli numbers $B_{n, \overline{\eta_h}}$ are defined by the generating series
$$
 \sum\limits_{a=1}^{N_h} \frac{\overline{\eta_h}(a)te^{at}}{e^{N_ht}-1} = \sum\limits_{n=0}^\infty B_{n, \overline{\eta_h}} \frac{t^n}{n!}.
$$
Write $N_h = N_1 N_2$ where $N_1$ contains only prime factors dividing $N$, and $N_2$ contains only prime factors not dividing $N$. Then there exist unique primitive Dirichlet characters $\chi_1$ and $\chi_2$ such that $\eta_h= \chi_1 \chi_2$. Since the conductor of $\chi$ is coprime to $N_2$, and $\eta_h$ is the Dirichlet character attached to $\chi \rho_h$, it follows that $\chi_2$ must be a quadratic Dirichlet character.

We now claim that $N_h^{-1/2}G(\eta_h) \in \O_p^\times$.  To see this, first note that $$G(\eta_h) = \mu G(\chi_1) G(\chi_2),$$  where $\mu$ is a root of unity. This follows by an elementary calculation using the Chinese remainder theorem, see, e.g., Section 1.6 of \cite{BerndtEvansWilliams1998}.  We know that $G(\chi_1) \overline{G(\chi_1)} = N_1$.  Since $G(\chi_1)$ is an algebraic integer, and $p$ does not divide $N_1$, it follows that $N_1^{-1/2}G(\chi_1) \in \O_p^\times$. Next, since $\chi_2$ is a quadratic character, it follows that $G(\chi_2)^2 = N_2$.  Hence, $N_2^{-1/2} G(\chi_2) \in \O_p^\times$. This proves our claim.

Putting it all together by (\ref{Dirichlet-L-value}), we have
$$
 \pi^{n-k} N_h^{k-n-\frac 12} L(k-n, \eta_h) = B_{k-n, \overline{\eta_h}} \ v,
$$
where $v \in \O_p^\times$.

Next, let us consider the generalized Bernoulli numbers. From page 137 of \cite{Leo}, if $N_h$ is divisible by two or more distinct primes, then $B_{k-n, \overline{\eta_h}}$ is an algebraic integer. Now, let $N_h = \ell^{e_\ell}$. If $\ell | 2N$, then by page 138-139 of \cite{Leo}, we see that $B_{k-n, \overline{\eta_h}} \in \mathcal{O}_p$. If $\ell \nmid 2N$, then we have $e_\ell = 1$ and $\eta_h$ is a quadratic character. Then again, by page 139 of \cite{Leo}, $B_{k-n, \overline{\eta_h}} \in \mathcal{O}_p$.
\end{proof}

%
%
%

\subsection{Proof of Proposition \ref{integrality-eis-series-thm}}\label{s:proof41}
We will now prove
\begin{equation}\label{req2new}
 \left(\frac{\det(h)}{C_h}\right)^{k-n-\frac 12} \prod\limits_{\ell \in {\bf C}} f_{h, \ell}(\chi(\ell)\ell^{-k}) \in \mathcal O_p,
\end{equation}
which will complete the proof of Proposition \ref{integrality-eis-series-thm}; see \eqref{twobrackets}. If $\ell \in {\bf C}$, and $p \neq \ell$, then $$f_{h, \ell}(\chi_\ell(\ell)\ell^{-k}) \in \mathcal O_p.$$ This follows immediately from the fact that $f_{h, \ell}$ has integer coefficients.
So \eqref{req2new}  follows from the next proposition.

\begin{proposition}\label{p:key}
 Let $p \nmid 2N$ be a prime, and let $h \in M^{\sym}_{2n}(\Z_p)$ be such that $\det(h) \neq 0$, $p\mid\det(h)$. Let $\C \simeq \overline{\Q_p}$ be an isomorphism, and let $v_p\colon \C \rightarrow \Q \cup \{ \infty\}$ be the resulting valuation, normalized so that $v_p(p) = 1$. Then $$\left(k-n-\frac 12\right)\left(v_p(\det(h)) - e_{p, h}\right) + v_p \left(f_{h, p}(\chi_p(p) p^{-k}) \right) \ge 0,$$ where we denote
 $$e_{p, h} := v_p(C_h) = \begin{cases} 0 & \text{ if } v_p(\det(h)) \text{ is even}, \\ 1 & \text{ if } v_p(\det(h)) \text{ is odd}.\end{cases}$$
\end{proposition}

The rest of this section will be concerned with the proof of Proposition \ref{p:key}. From now on, we let $p$, $h$ be as in that proposition.

We need to recall the definition of the local Siegel series. Given a matrix $R \in \GL_{2n}(\Q_p)$, it is well-known by the theory of elementary divisors that there exist matrices $A, B \in \GL_{2n}(\Z_p)$ such that
$$
 ARB = \mathrm{diag}(p^{e_1}, p^{e_2}, \ldots, p^{e_{2n}}),
$$
where the $e_i$ are integers independent of the choice of $A$, $B$, and $e_1 \le e_2 \ldots \le e_{2n}.$ We define the integer $e(R)\ge 0$ by
$$
 e(R) = -\sum_{i: e_i \le 0} e_i.
$$
Now given any matrix $h \in M^{\sym}_{2n}(\Z_p)$, $\det(h) \neq 0$ define a formal power series $B_p(X, h)$ and a formal $p$-Dirichlet series $b_p(s, h)$ by
$$
 B_p(X, h) = \sum_{R \in M^{\sym}_{2n}(\Q_p)/ M^{\sym}_{2n}(\Z_p)} X^{e(R)}e^{2 \pi i \rm{tr}(hR)}, \quad b_p(s, h) = B_p(p^{-s}, h).
$$
The series $b_p(s, h)$ is known as the local Siegel series. It can be shown that it has a closed form expression as a rational function in $p^{-s}$, and hence is defined on the entire complex plane with finitely many poles.  Unpacking the notation of Chapter 16 of \cite{shibook2}, we find
\begin{align}
 f_{h,p}(\chi_p(p) p^{-k}) &= \frac{\Lambda_p(k/2)}{L_p(k-n, \chi \rho_h)} B_p(\chi(p)p^{-k}, h)\nonumber\\ \label{e:fTsiegel1} &= \frac{\Lambda_p(k/2)}{L_p(k-n, \chi \rho_h)} b_p(k+it, h),
\end{align}
where we define $t\in \R$ by $\chi_p(p) = p^{-it}$. Hence
$$
 v_p\left(f_{h,p}(\chi_p(p) p^{-k})\right) = v_p\left(b_p(k+it, h) \right) + v_p\left(\Lambda_p(k/2)\right) - v_p\left(L_p(k-n, \chi \rho_h)\right).
$$
Note that $L_p(k-n, \chi\rho_h)=1$ if $e_{p, h}=1$ and $v_p\left(L_p(k-n, \chi \rho_h)\right) = k-n$ if $e_{p, h}=0$. Using \eqref{LambdaNdefeq}, we deduce that
\begin{equation}\label{e:firstineq}v_p\left(f_{h,p}(\chi_p(p) p^{-k})\right) = v_p\left(b_p(k+it, h) \right) + 2nk-n^2 + e_{p, h}(k-n).
\end{equation}
On the other hand, Kitaoka \cite{kitaoka84} found an exact formula for $b_p(s, h)$. The following result follows by substituting $s = k+it$ in Theorem 2 of \cite{kitaoka84}.

\begin{proposition}[Theorem 2 of \cite{kitaoka84}]
 Let $p \nmid N$ be an odd prime, and let $h \in M^{\sym}_{2n}(\Z_p)$  such that $\det(h)\neq 0$ and $p\mid\det(h)$. Then \begin{equation}\label{e:kitaokaformula}
  b_p(k+it, h) = \sum_{\substack{G \in \GL_{2n}(\Z_p) \bs M_{2n}(\Z_p) \\ 2v_p(\det(G)) \le v_p(\det(h))}} \chi_p^2(p)p^{v_p(\det(G))(2n+1-2k)} \alpha_{\chi_p}(-{}^tG^{-1}hG^{-1}, k),
 \end{equation}
 where $\alpha_{\chi_p}(S, k)$ is defined as follows. If $S \notin M^{\sym}_{2n}(\Z_p)$ then $\alpha_{\chi_p}(S, k)=0$. If  $S \in M^{\sym}_{2n}(\Z_p)$, then let $N_S$ denote the space $\mathbb{F}^{2n}_p$ equipped with the quadratic form defined by $S$. Write $N_S = N_1 \perp N_2$ where $N_2$ is the maximal totally singular subspace. Let $d=\dim(N_1)$. Put $\epsilon=1$ if $N_1$ is an orthogonal direct sum of hyperbolic planes, or if $d=0$; otherwise $\epsilon=-1$. Then
 \begin{equation}
  \alpha_{\chi_p}(S,k) = \begin{cases} (1-\chi_p(p)p^{-k})(1+\epsilon \chi_p(p)p^{2n-d/2-k})\prod_{1\le i \le 2n- d/2-1} (1-\chi^2_p(p)p^{2i-2k}) & \text{ if } 2\mid d, \\
  (1-\chi_p(p)p^{-k}) \prod_{1\le i \le 2n- (d+1)/2} (1-\chi^2_p(p)p^{2i-2k}) & \text{ if } 2 \nmid d.\end{cases}
 \end{equation}
\end{proposition}

\begin{remark}In \cite{Bocherer1984}, B\"ocherer found a very similar formula.
\end{remark}

\begin{remark}\label{d-det(G)-rmk}
 Note that $p^{{\rm dim}(N_2)}\mid\det(S)$ above. For the special case $S = -{}^tG^{-1}hG^{-1}$, we get
 $$
  v_p(\det(G)) \leq \left\lfloor \frac 12 v_p(\det(h)) - n + \frac d2 \right\rfloor.
 $$
\end{remark}

Using $v_p(x, y) \geq {\rm min}(v_p(x), v_p(y))$ and taking valuations of both sides of  \eqref{e:kitaokaformula}, we see that
\begin{equation}\label{e:kitaoconseq1}
 v_p(b_p(k+it, h)) \ge \mathrm{min}_G \left(v_p(\det(G))(2n+1-2k) + v_p(\alpha_{\chi_p}(-{}^tG^{-1}hG^{-1}, k)) \right).
\end{equation}
Above and henceforth, $G$ denotes any element in $M_{2n}(\Z_p)$ such that $2v_p(\det(G)) \le v_p(\det(h))$. Note from the formulas for $\alpha_{\chi_p}(S,k)$ that whenever $\alpha_{\chi_p}(-{}^tG^{-1}hG^{-1}, k)$ is non-zero,
\begin{equation}\label{e:kitcon2}
 v_p(\alpha_{\chi_p}(-{}^tG^{-1}hG^{-1}, k)) \geq \begin{cases}\lfloor (2n-d/2) (2n-d/2-2k) \rfloor &\text{ if } d \ge 4n-2k, \\
 -k^2 &\text{ if }  d < 4n-2k. \end{cases}
\end{equation}
We observe here that the case $d<4n-2k$ can only occur for $k$ in the range $n+1 \le k \le 2n-1$.

By Remark \ref{d-det(G)-rmk} and \eqref{e:kitcon2} (and using $k \geq n+1$) the  contribution to the right hand side of (\ref{e:kitaoconseq1}) from the matrices $G$ for which $d \ge 4n-2k$ is always greater than or equal to
\begin{equation}\label{minestimate1}
 \left\lfloor\frac 12 v_p(\det(h)) - n + \frac d2 \right\rfloor (2n+1-2k) + \lfloor (2n-d/2) (2n-d/2-2k) \rfloor.
\end{equation}
One can check that in all cases, $v_p(\det(h))$ even or odd, $d$ even or odd, \eqref{minestimate1} is greater or equal to
\begin{equation}\label{minestimate2}
 (v_p(\det(h))-e_{p, h})(n-k+1/2) + n(n-2k) +e_{p, h}(n-k).
\end{equation}

Next, we consider the matrices $G$ for which $d < 4n-2k$. Again, by Remark \ref{d-det(G)-rmk} and \eqref{e:kitcon2}, the  contribution to the right hand side of (\ref{e:kitaoconseq1}) from the matrices $G$ for which $d < 4n-2k$ is greater than or equal to
\begin{equation}\label{minestimate3}
 \left\lfloor \frac 12 v_p(\det(h)) - n + \frac d2 \right\rfloor (2n+1-2k) -k^2.
\end{equation}
One can check that in all cases \eqref{minestimate3} is greater or equal to
\begin{equation}\label{minestimate4}
 (v_p(\det(h))-e_{p, h})(n-k+1/2) + n(n-2k) +e_{p, h}(n-k).
\end{equation}
Combining the above, we get
\begin{equation}\label{e:kit5}
 v_p(b_p(k+it, T)) \ge (v_p(\det(h))-e_{p, h})(n-k+1/2) + n(n-2k) +e_{p, h}(n-k).
\end{equation}
Proposition \ref{p:key} follows by combining \eqref{e:kit5} and \eqref{e:firstineq}. This completes the proof of Proposition~\ref{integrality-eis-series-thm}.
\subsection{An application of the \texorpdfstring{$q$}{}-expansion principle}\label{s:qexp}

In this short subsection, we continue to focus on the case $m_0=0$ and extend Proposition \ref{integrality-eis-series-thm} from $h = \iota$ to general $h$ and also include the case $N=1$.
\begin{proposition}\label{integrality-eis-series-thm-new}
 Let $k \ge n+1$. If $k=n+1$, assume further that $\chi^2 \neq 1$. Let $p$ be any prime such that $p \nmid 2N$ and $p \ge 2k$. Let $h \in \Sp_{4n}(\hat{\Z})$. Then
 \begin{equation}
  \pi^{n + n^2-(2n+1)k} \Lambda^N\Big(\frac k2\Big) E_{k,N}^{\chi}(Z,0; h) \in M_k(\Gamma_{4n}(N),\mathcal{O}_p).
 \end{equation}
\end{proposition}

 Before proving this proposition, let us begin with some related discussion. Let $P$ be the Siegel parabolic subgroup of $\Sp_{2m}$. By a \emph{cusp} of $\Gamma(N)$ we mean an element of the (finite) double coset space
\begin{equation}\label{step2eq2}
 \Gamma(N)\backslash\Sp_{2m}(\Q)/P(\Q)\cong\Gamma(N)\backslash\Sp_{2m}(\Z)/P(\Z).
\end{equation}
Let $\gamma\in P(\Z)$. Let $F\in M_k(\Gamma(N))$ and assume that the ring $R$ contains $e^{2\pi i/N}$. It is easy to see that
\begin{equation}\label{step2eq3}
 F\in M_k(\Gamma(N),R)\qquad\Longleftrightarrow\qquad F|_k\gamma\in M_k(\Gamma(N),R).
\end{equation}
Hence, for given $F\in M_k(\Gamma(N),R)$ and $g\in\Sp_{2m}(\Z)$, whether $F|_kg$ also has Fourier coefficients in $R$, depends only on the double coset $\Gamma(N)gP(\Z)$, i.e., on the cusp determined by $g$. In fact, provided that $R$ contains $1/N$, the following much stronger statement is true.
\begin{proposition}\label{qexpansiontheorem}
 Let $N$ be an integer, and let $k,m$ be positive integers.  Let $F\in M_k(\Gamma(N))$ be a Siegel modular form of degree $m$ and weight $k$ with respect to the principal congruence subgroup $\Gamma(N)$ of $\Sp_{2m}(\Z).$ Let $R$ be a subring of $\C$ containing $\Z[e^{2\pi i /N}, e^{2\pi i /3}, 1/N, 1/6]$. Then, for any $g\in\Sp_{2m}(\Z)$,
 \begin{equation}\label{step2eq4}
  F\in M_k(\Gamma(N),R)\qquad\Longleftrightarrow\qquad F|_kg\in M_k(\Gamma(N),R).
 \end{equation}
\end{proposition}
\begin{proof}
This follows from the \emph{$q$-expansion principle}; see Proposition 1.5 of \cite{ChaiFaltings1990}, Sect.~2.1 of \cite{Ichikawa2013}, and Sect.~5.5.6 of \cite{SkinnerUrban2014}.
\end{proof}

We can now finish the proof of Proposition \ref{integrality-eis-series-thm-new}.  We proved in Proposition \ref{integrality-eis-series-thm} that if $N>1$, and $p$ is a prime such that $p \nmid 2N$ and $p \ge 2k$, then
\begin{equation}\label{step2eq8}
  \pi^{n + n^2-(2n+1)k} \Lambda^N\Big(\frac k2\Big) E_{k,N}^{\chi}(Z,0; \iota) \in M_k(\Gamma_{2m}(N),\mathcal{O}_p).
\end{equation}

It is straightforward to verify that if $h$ is an element of $\Sp_{4n}(\Z)$, considered as an element of $\prod_{p<\infty} \Sp_{4n}(\Z_p)$, then
\begin{equation}\label{step2eq9}
 E^\chi_{k,N}(Z,0;h)=E^\chi_{k,N}(Z,0; \iota)|_k(\iota^{-1}h).
\end{equation}
Putting $h=1$ above, and using Proposition \ref{qexpansiontheorem} we get that $$\pi^{n + n^2-(2n+1)k} \Lambda^N\Big(\frac k2\Big) E_{k,N}^{\chi}(Z,0; 1) \in M_k(\Gamma_{4n}(N),\mathcal{O}_p).$$

We can now extend Proposition \ref{integrality-eis-series-thm} to include the case $N=1$. Indeed, by Proposition~2.4 of \cite{shi83},
$$
 E_{k,1}^{\bf 1}(Z,0; 1) = \sum_\tau E_{k,2}^{\bf 1}(Z,0; 1) |_k \tau,
$$
where $\tau$ ranges over a finite subset of $\Sp_{4n}(\Z)$. This shows that Proposition \ref{integrality-eis-series-thm-new} holds in the case $h=1$ (and $N$ arbitrary).

Now, another application of Proposition \ref{qexpansiontheorem} via \eqref{step2eq9}, and observing that $\Sp_{4n}(\Z)$ is dense in $\Sp_{4n}(\hat\Z)$ by strong approximation, we get the proof of Proposition \ref{integrality-eis-series-thm-new}.
\section{Application of differential operators} \label{s:final}
\subsection{An integrality result for the Maass operator}
In this subsection, $m$ will be any positive integer; we will take $m=2n$ in the next subsection. The Maass operator $\Delta_k$ of weight $k$ and degree $m$ is defined on smooth functions $f:\H_m\to\C$ by the following formulas,
\begin{equation}\label{Deltadefeq1}
 \Delta=\det(\partial_{ij}),\qquad
 \partial_{ij}=\begin{cases}
                \displaystyle\frac{\partial}{\partial Z_{ij}}&\text{if }i=j,\\[3ex]
                \displaystyle\frac12\frac{\partial}{\partial Z_{ij}}&\text{if }i\neq j,
               \end{cases}
\end{equation}
\begin{equation}\label{Deltadefeq2}
 (\Delta_kf)(Z)=\det(Z-\bar Z)^{\kappa-k-1}\Delta\Big(\det(Z-\bar Z)^{k-\kappa+1}f(Z)\Big),\qquad\kappa=\frac{m+1}2.
\end{equation}
If $\Gamma$ is a congruence subgroup of $\Sp_{2m}(\Q)$, $f\in M_k(\Gamma)$, then $\Delta_kf\in N_{k+2}(\Gamma)$; see (7.3) of \cite{Shimura1981} or \S19 of \cite{maassbook}.  More generally, $\Delta_k$ takes $N_k(\Gamma)$ to $N_{k+2}(\Gamma)$; see \cite[4.11]{Shimura1987}.

The relation with the operator $\delta_k$ defined in (1.2) of \cite{Pan2005} is
\begin{equation}\label{deltaDeltaeq2}
 \delta_k=\Big(-\frac i2\Big)^m\pi^{-m}\Delta_k.
\end{equation}
For a positive integer $r$, let $\Delta_k^{r}$ and $\delta_k^{r}$ be the $r$-fold iterations of these operators, i.e.,
\begin{equation}\label{Deltardefeq}
 \Delta_k^{r}=\Delta_{k+2(r-1)}\circ\ldots\circ\Delta_{k+2}\circ\Delta_k,\qquad
 \delta_k^{r}=\delta_{k+2(r-1)}\circ\ldots\circ\delta_{k+2}\circ\delta_k.
\end{equation}

From \eqref{deltaDeltaeq2}, we get
\begin{equation}\label{deltaDeltaeq3}
 \delta_k^{r}=\Big(-\frac i2\Big)^{mr}\pi^{-mr}\Delta_k^{r}.
\end{equation}
Define $\lambda_i : M_m(\C) \rightarrow \C$, for $0 \leq i \leq m$, by
\begin{equation}\label{lambda-defn}
 \det(tI_m + Z) = \sum\limits_{i=0}^m \lambda_i(Z)t^{m-i}.
\end{equation}
Let $Y$ be an invertible $m\times m$ matrix, and $\xi$ be any $m\times m$ matrix. Letting $Z=\xi Y$, we get
\begin{equation}\label{lambda-defn2}
 \det(tY^{-1} + \xi) = \sum\limits_{i=0}^m \frac{\lambda_i(\xi Y)}{\det(Y)}t^{m-i}.
\end{equation}
It follows that
\begin{equation}\label{lambda-defn3}
 \frac{\lambda_i(\xi Y)}{\det(Y)}\text{ is a polynomial in the entries of $Y^{-1}$ of degree at most }m-i.
\end{equation}
We also see that
\begin{equation}\label{lambda-defn4}
 \text{if $\xi$ has entries in a ring $R$, then the coefficients of the polynomial }\frac{\lambda_i(\xi Y)}{\det(Y)}\text{ are also in }R.
\end{equation}
Let $F \in M_k(\Gamma_{2m}(N))$ with Fourier expansion
$$
 F(Z) = \sum\limits_{Q \in N^{-1}L} c(Q)e^{2 \pi i {\rm tr}(QZ)},
$$
where $L$ is the set of half-integral, symmetric, positive definite $m\times m$ matrices. Theorem 3.6(a) of \cite{Pan2005} gives the following Fourier expansion of $\delta_k^{r}F$,
$$
 \delta_k^{r}F(Z) = \sum\limits_{Q \in N^{-1}L} c(Q) \sum\limits_{t=0}^r \binom{r}{t} \det(Q)^{r-t}\sum\limits_{|\hat{L}| \leq mt-t} R_{\hat{L}}\Big(\frac{m+1}2 - k -r\Big)\prod\limits_{j=1}^t \frac{\lambda_{l_j}(4 \pi Q Y)}{\det(4 \pi Y)}e^{2 \pi i {\rm tr}(QZ)},
$$
where $\hat{L}$ runs over all multi-indices $0 \leq l_1 \leq \cdots \leq l_t \leq m$, such that $|\hat{L}| = l_1 + \cdots + l_t \leq mt-t$ and the coefficients $R_{\hat{L}}$ are polynomials with coefficients in $\Z[1/2]$. Note that Theorem 3.6(a) of \cite{Pan2005} is stated for Siegel congruence subgroups, but the proof goes through without change for principal congruence subgroups. By \eqref{lambda-defn3} and \eqref{lambda-defn4}, $\frac{\lambda_{l_j}(4 \pi Q Y)}{\det(4 \pi Y)}$ is a polynomial in the entries of $4\pi Y$ of degree at most $m-l_j$ and with coefficients in $\Z[1/N]$. It follows that  $\delta_k^{r} F \in N_{k+2r}(\Gamma_{2m}(N), \mathcal{O}_p)$, as long as $p\nmid2N$. In view of \eqref{deltaDeltaeq3}, we get the following result.

\begin{proposition}\label{step3prop}
 Let $m,k,N,r$ be positive integers. Let $p$ be a prime not dividing $2N$. Then, for any $F \in M_k(\Gamma_{2m}(N), \mathcal{O}_p)$,
 $$
  \pi^{-mr}\Delta_k^{r} F \in N_{k+2r}(\Gamma_{2m}(N), \mathcal{O}_p).
 $$
\end{proposition}
\subsection{Completion of the proof of Theorem \ref{main-Eis-arith-thm}}
We can now complete the proof of part 2 of Theorem \ref{main-Eis-arith-thm}.
\begin{proposition}\label{step4prop}
 Suppose that $0 \le m_0 \le \frac{k-n-1}{2}$ is an integer  and $h \in \Sp_{2m}(\hat{\Z})$. If $m_0 = \frac{k-n-1}{2}$, assume further that $\chi^2 \neq 1$.  Then, for any prime $p$ with $p\nmid2N$ and $p\geq 2k$,
 $$
\pi^{n+n^2-(2n+1)k+(2n+2)m_0} \Lambda^N\Big(\frac{k-2m_0}2\Big)E_{k,N}^{\chi}(Z,-m_0; h) \in N_k(\Gamma_{4n}(N), \mathcal{O}_p).
 $$
 Here, $\Lambda^N$ is the factor defined in \eqref{LambdaNdefeq}.
\end{proposition}
\begin{proof}
The operator $\Delta_k^{r}$ defined in \eqref{Deltardefeq} coincides with the operator $\Delta_k^r$ defined in the proof of Theorem~17.9 of \cite{shibook2}; see Proposition 7.2 of \cite{Shimura1981}. Setting $p=m_0$, $q=k-2m_0$ and $s=\frac{k-2m_0}2$ in (17.20) of \cite{shibook2}, and observing the shift \eqref{step1eq1}, we get
\begin{equation}\label{step4eq1}
 d^{-1}(-4)^{nm_0}\cdot\Delta^{m_0}_{k-2m_0}(E_{k-2m_0,N}^\chi(Z, 0;h)) =  E_{k,N}^\chi(Z, -m_0; h)
\end{equation}
with
\begin{equation}\label{step4eq2}
 d=\prod_{a=1}^{2n}\,\prod_{b=1}^{m_0}\Big(2m_0-k-b+\frac{a+1}2\Big).
\end{equation}
The condition $p\geq 2k$ assures that $p\nmid d$.

By Proposition \ref{integrality-eis-series-thm-new},
\begin{equation}\label{step2eq10b}
 \pi^{n+n^2-(2n+1)(k-2m_0)} \Lambda^N\Big(\frac{k-2m_0}2\Big) E_{k-2m_0,N}^{\chi}(Z,0;h) \in M_{k-2m_0}(\Gamma_{4n}(N),\mathcal{O}_p).
\end{equation}
Hence the assertion follows from Proposition \ref{step3prop}.
\end{proof}

%
%
%
\addcontentsline{toc}{section}{Bibliography}
\bibliography{eisenstein}{}
\end{document}